\documentclass{amsart}
\usepackage[utf8]{inputenc}
\usepackage{tikz-cd}
\usepackage{comment}

\usepackage{amsfonts, amsthm,amssymb}
\usepackage{amsmath}

\newtheorem{proposition}{Proposition}[section]
\newtheorem{lemma}[proposition]{Lemma}
\newtheorem{corollary}[proposition]{Corollary}
\newtheorem{theorem}[proposition]{Theorem}
\newtheorem{conjecture}[proposition]{Conjecture}
\theoremstyle{definition}
\newtheorem{remark}[proposition]{Remark}
\newtheorem{definition}[proposition]{Definition}
\newtheorem{example}[proposition]{Example}

\def\Z{\mathbb{Z}}

\def\G{\Gamma}

\def\C{\mathcal{C}}

\def\Q{\mathcal{Q}}

\usepackage{mathtools}

\begin{document}
\title[The extended irregular domination problem]{The extended irregular domination problem}

\author[L. Mella]{Lorenzo Mella}
\address{Dip. di Scienze Fisiche, Informatiche, Matematiche, Universit\`a degli Studi di Modena e Reggio Emilia, Via Campi 213/A, I-41125 Modena, Italy}
\email{lorenzo.mella@unipr.it}
\author[A. Pasotti]{Anita Pasotti}
\address{DICATAM - Sez. Matematica, Universit\`a degli Studi di Brescia, Via
Branze 43, I-25123 Brescia, Italy}
\email{anita.pasotti@unibs.it}


\begin{abstract}
In this paper we introduce a new domination problem strongly related to the following  one recently proposed by
Broe, Chartrand and Zhang. One says that a vertex $v$ of a graph $\G$ labeled with an integer $\ell$ dominates the vertices of $\G$ having distance $\ell$ from $v$. An irregular dominating set of a given graph $\G$ is  a set $S$ of vertices of $\G$, having distinct positive labels, whose elements dominate every vertex of $\G$. Since it has been proven that no connected vertex transitive graph admits an irregular dominating set, here we introduce the concept of an \emph{extended} irregular dominating
set, where we admit that precisely one vertex, labeled with 0, dominates itself. Then we present existence or non existence results of an extended irregular dominating set $S$ for several classes of graphs, focusing in particular on the case in which $S$ is as small as possible. We also propose two conjectures.
\end{abstract}

\maketitle

\section{Introduction}
The concept of domination plays an important role in graph theory, having a large variety of applications and being closely related to other topics in graphs, such as, just to make an example, independent sets. This area began with the work of Berge \cite{B} and Ore \cite{O}, but it becomes active only 15 years later thanks to the survey by Cockayne and Hedetniemi \cite{CH}.

In the following, by $\Gamma$ we denote an undirected simple graph with vertex set $V$ and edge set $E$. A set $S \subseteq V$ is said to be a \textit{dominating set} if for every $u \in V$ there exists a vertex $v\in S$ such that $\{u,v\}\in E$. Over the course of the years, various kinds of domination have been introduced and investigated, see \cite{HHH}. In this work, we propose a new variation of a domination problem first considered in \cite{BCZ} and further studied in \cite{ACZ, BCZ21, CZ, MZ22,MZ}, that we report in what follows. 

  A vertex $v$ of $\Gamma$, labeled with a positive integer $\ell$, is said to \textit{dominate}, or \textit{cover}, the vertices of $\Gamma$ having distance $\ell$ from $v$. Similarly, a set $S \subseteq V$ of labeled vertices is said to \textit{dominate} (or \textit{cover}) a vertex $u \in V$ if there exists a  vertex $v\in S$ covering $u$. Then, an \textit{irregular dominating set} of $\G$ is a set $S$ of vertices of $\G$ having distinct positive labels that covers every vertex of $V$. Note that the definition does not require that the labelings are consecutive integers.
As it is standard in the topic of domination, it is interesting to determine the minimum value of $k$ such that an irregular dominating set of cardinality $k$ exists, such a value is denoted by $\gamma_i(\Gamma)$ and it is called the \emph{irregular domination number} of $\G$.

By a counting argument, in \cite{CZ} it has been proven that no connected vertex transitive graph admits an irregular dominating set. However, due to the many symmetries and desirable properties of vertex transitive graphs, it is natural to wonder if there exists a labeling of the vertices of such a graph, different from the standard one, so that it is possible to obtain something similar to an irregular dominating set. Here, we propose the concept of an extended irregular dominating set of a graph, where we admit that precisely one vertex, labeled with $0$, dominates itself. More formally, we give the following new definition.
\begin{definition}
Let $\Gamma= (V,E)$ be an undirected graph and let $k\geq 0$. A $k$-\textit{extended irregular dominating set} $S \subseteq V$ is a set of $k$  vertices having distinct non-negative labelings that covers every vertex of $V$. A labeling $\lambda$ realizing the covering property for $S$ is said to be a $k$-\textit{extended irregular dominating labeling}.
\end{definition}
Clearly, if $\lambda(v)\neq 0$ for every $v\in S$ then we find again the concept of an irregular dominating set/labeling, hence in the following we always assume that there exists a vertex $v\in S$ labeled by $0$ which dominates itself.

We have chosen the term  ``\emph{extended} irregular dominating labeling", since the connection between these and classical irregular dominating labelings is very similar to the one between extended and classical Skolem sequences, see \cite{S} for further details.

We call the minimum cardinality of an extended irregular dominating set of $\G$, denoted by $\gamma_e(\G)$, the \textit{extended irregular domination number} of $\G$.

\begin{example} \label{ex:ext_dom_set}
Here we show an extended irregular dominating labeling inducing an extended irregular dominating set for the cycle of length $6$ and for the cube:
\begin{center}
\includegraphics[width = 0.6\textwidth]{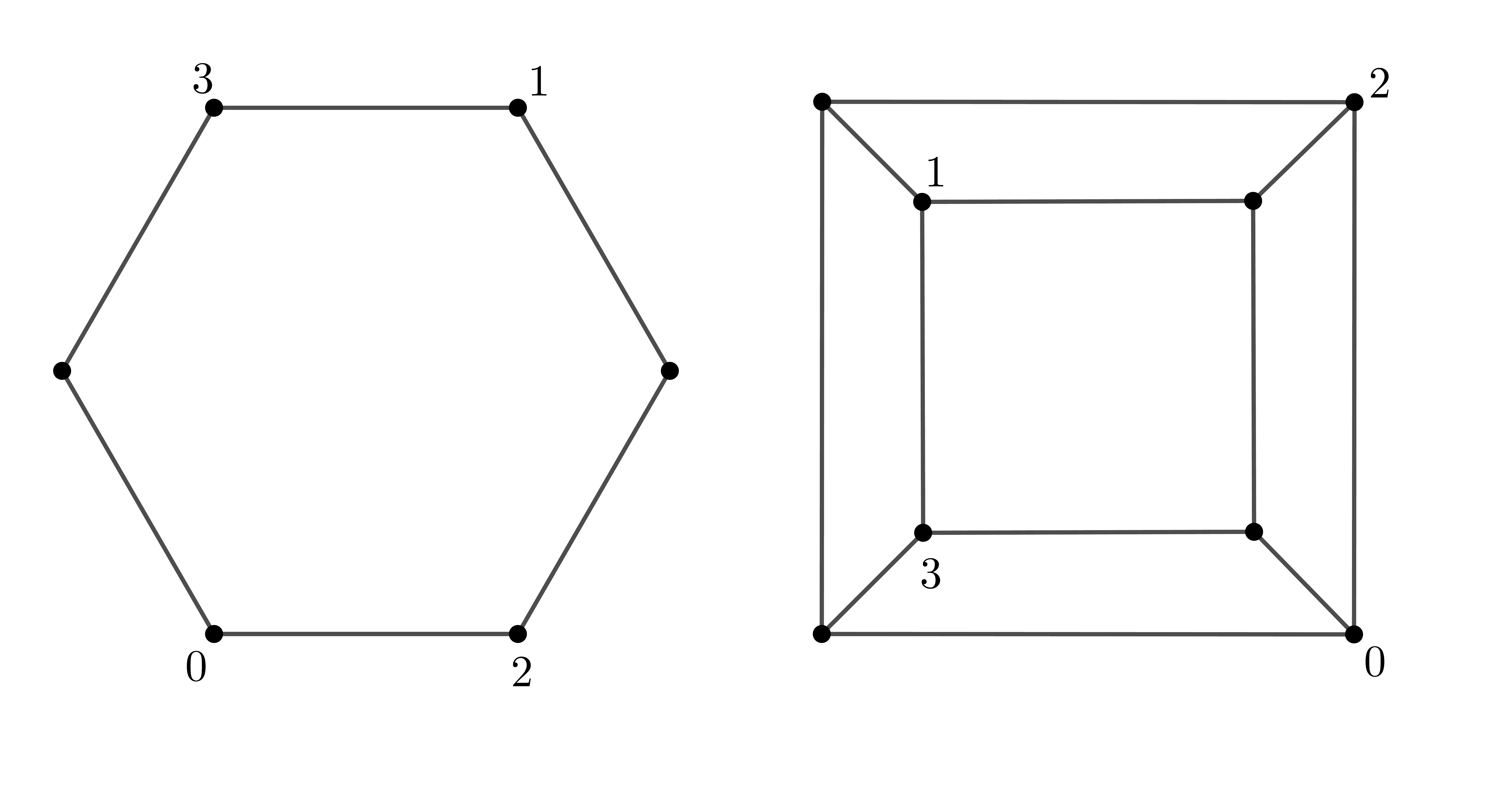}
\end{center}
\end{example}
We underline that a graph does not necessarily admit an extended irregular dominating labeling, take for example the cycle of length three. Clearly, if such a labeling exists, the most interesting problem is that of determining a $k$-extended irregular dominating labeling for a given graph $\G$ with $k$ as small as possible, or in other words that of determining $\gamma_e(\G)$.
We say that a $k$-extended irregular dominating set (labeling, respectively) is \emph{optimal} if $k=\gamma_e(\G)$, that is if there is no $k'$-extended irregular dominating set (labeling, respectively) with $k'<k$. One can easily check that the labelings of Example \ref{ex:ext_dom_set} are optimal.

The paper is organized as follows. In Section \ref{sec:2} we determine $\gamma_e(\G)$ for every vertex-transitive graph $\G$. This allows us to show that, for this class of graph, an extended irregular dominating set is necessarily optimal.  Then, in Section \ref{sec:radius}, we prove the existence or non existence of an optimal extended irregular dominating set for several classes of vertex-transitive graphs. In Section \ref{sec:cycles} we show that the existence of an optimal extended irregular dominating labeling of a cycle of odd length $n$ is equivalent to the existence of a strong starter of $\Z_n$. Then we present some results for an optimal extended irregular dominating labeling of odd cycles obtained as a consequence of known results on strong starters, as well as some new results in the case of cycles of single even length. In Section \ref{sec:paths} we focus on a class of non vertex-transitive graphs: the paths. We point out that, in general, it is not easy to establish the value of $\gamma_e(\G)$ if $\G$ is a non vertex-transitive graph. Here we firstly present a complete answer to the existence problem for an  extended irregular dominating set for this class of graphs, then we give a lower bound for  $\gamma_e(\G)$, $\G$ being a path, and then we establish when this bound is reach. To conclude, in the last section, we propose two conjectures: the first one about the existence of an optimal extended irregular dominating labeling of cycles of length divisible by $4$, while the second regards the value of $\gamma_e(\G)$ when $\G$ is a path.

\section{Preliminary results}\label{sec:2}
In this section, we show that an extended irregular dominating set of a vertex-transitive graph, if it exists, is necessarily optimal.

Firstly we need to introduce some notation and to recall some basic concepts of graph theory. Given two integers $a,b$ with $a\leq b$, by $[a,b]$ we mean the set $\{a,a+1,\ldots,b\}$. The \textit{degree} of a vertex $v$ of $\G$, denoted by $\mathrm{deg}(v)$, is the number of neighbours of $v$ in $\G$. A graph is said to be \textit{regular} if all its vertices have the same degree.
We also recall that the \textit{diameter} of a graph $\Gamma$, denoted by $\mathrm{diam}(\Gamma)$, is the largest distance between any pair of vertices of $\G$. 
\begin{remark} \label{rem:radius_diam}
The labels of an extended irregular dominating set of a graph $\Gamma$, if it exists, can assume value in $[0,\mathrm{diam}(\Gamma)]$
and hence $\gamma_e(\G)\leq \mathrm{diam}(\Gamma)$+1. 
\end{remark}

\begin{proposition}\label{prop:k12}
    A $k$-extended irregular dominating set cannot exist for $k=2,3$.
\end{proposition}
\begin{proof}
    When $k=2$, if the label different from $0$ is assigned to a vertex $v$, then it is not possible to dominate $v$ with the remaining label $0$.

    Suppose now $k=3$ and let $v$ and $w$ be the vertices with a non-zero label. It is easy to see that necessarily $w$ has to be dominated by $v$ and vice versa, but this is not possible since we use distinct labelings. 
\end{proof}

\begin{corollary}\label{cor:completo}
    The complete graph $K_n$, with $n>1$, and the complete bipartite graph $K_{m,n}$ do not admit an extended irregular dominating set.
\end{corollary}

Clearly, a graph has a $1$-extended irregular dominating set if and only if it is an isolated vertex. Hence, if a non trivial graph $\G$ admits a $4$-extended irregular dominating set
$S$, then $S$ is optimal and $\gamma_e(\G)=4$.

With the following lemma we show an interesting property of vertex-transitive graphs, that implies the well-known result that these graphs are regular.
\begin{lemma} \label{lem:radius_vtx_tran}
Let $\Gamma = (V,E)$ be a vertex-transitive graph. For every vertex $v\in V$ and for every $i \in [0,\mathrm{diam}(\Gamma)]$ let $s_i(v)$ be the  number of vertices having distance $i$ from $v$. Then, the sequence $(s_0(v), \dotsc, s_{\mathrm{diam}(\Gamma)}(v))$ is invariant on the choice of $v$, and $\sum s_i(v) = |V|$. 
\end{lemma}

\begin{theorem} \label{prop:disjoint}
Let $\Gamma = (V,E)$ be a vertex-transitive graph admitting an extended irregular dominating set $S$. Then, for every vertex $u \in V$ there exists a unique vertex $v \in S$ covering $u$.
\end{theorem}
\begin{proof}
 From Remark \ref{rem:radius_diam} we have that an extended irregular dominating labeling of $\Gamma$ takes values in $[0,\mathrm{diam}(\Gamma)]$.
Let $S$ be an extended irregular dominating set and $W = \{w_1,\dotsc, w_a\}$ be the set of vertices that are covered by at least two vertices  of $S$. We have to prove that $W=\emptyset$. For every $j \in [1,a]$, let $m_j \geq 2$ be the number of vertices of $S$ covering $w_j$. Then, by Lemma \ref{lem:radius_vtx_tran}, the number of vertices covered by $S$ is given by:
\[
|V| - \sum_{j=1}^a (m_j-1) \leq |V| - a.
\]
Since $S$ is an extended irregular dominating set, we deduce that necessarily $a=0$, hence every vertex of $\Gamma$ is covered by precisely one vertex of $S$.
\end{proof}
As a consequence, we have that if there exists a $k$-extended irregular dominating set of a vertex-transitive graph $\G$, then $\G$ cannot admits a $k'$-extended irregular dominating set with $k'\neq k$, otherwise at least one vertex should be dominated more than once. By the above arguments and by Lemma \ref{lem:radius_vtx_tran} we have the following.
\begin{corollary}\label{cor:k=radius}
A $k$-extended irregular dominating set $S$ of a vertex-transitive graph $\G$, if it exists, is optimal and $k=\gamma_e(\G)=\mathrm{diam}(\G)+1$. Moreover, the $k$ vertices of $S$ are labeled with all the elements in $[0,\mathrm{diam}(\G)]$. 
\end{corollary}

\section{Some results obtained using the diameter of the graph}\label{sec:radius}
In what follows, we show some existence and non-existence results of extended irregular dominating sets for vertex-transitive graphs in which the diameter of the graph plays a crucial role. 

First note that as a consequence of Proposition \ref{prop:k12}, we have the following.
\begin{proposition}
Let $\Gamma$ be a vertex-transitive graph such that $\mathrm{diam}(\Gamma) \in \{1,2\}$. Then, there does not exist  an extended irregular dominating set in $\Gamma$.  
\end{proposition}

We consider now graphs with diameter equal to three. Firstly, we recall the definition of the crown graph. Let $K_{n,n}$ be the complete regular bipartite graph on $2n$ vertices, and denote by $A = \{a_1,a_2,\dotsc, a_n\}$ and $B = \{b_1,b_2,\dotsc, b_n\}$ the two parts of $K_{n,n}$. Note that $M = \{\{a_i,b_i\}: i \in [1,n]\}$ is a perfect matching of $K_{n,n}$. Given an integer $n\geq 3$, the \textit{crown graph of order $2n$} is the graph $\Gamma = K_{n,n}\setminus M$. We recall that these graphs are vertex-transitive, and have diameter equal to three. We are going to prove that this is the unique class of bipartite vertex-transitive graphs with diameter $3$ admitting an extended irregular dominating set.
\begin{proposition}
The crown graph  admits an optimal extended irregular dominating set.
\end{proposition}
\begin{proof}
Let $\Gamma$ be a crown graph, and pick  any two pairs of non-adjacent vertices $(a_i,b_i)$ and $(a_j,b_j)$, and let $\lambda: \{a_i,a_j,b_i,b_j\} \rightarrow [0,3]$ be the following labeling:
$$\lambda(a_i) = 0,\ \ \lambda(a_j) =3,\ \ \lambda(b_i) = 1,\ \  \lambda(b_j) =2.$$
We have that $b_i$ dominates $A \setminus \{a_i\}$, while  $b_j$ dominates $B \setminus \{b_j\}$. Trivially, $a_i$ and $a_j$  dominate itself and $b_j$, respectively, thus concluding the proof.
\end{proof}

\begin{theorem}
Let $\Gamma$ be a vertex-transitive bipartite graph with $\mathrm{diam}(\Gamma) = 3$. Then, $\Gamma$ admits an optimal extended irregular dominating set if and only if it is a crown graph.
\end{theorem}
\begin{proof}
Let $\Gamma $ be a vertex-transitive bipartite graph with $\mathrm{diam}(\Gamma) = 3$, and assume that it admits a $k$-extended irregular dominating labeling. Recall that, by Corollary \ref{cor:k=radius}, we have $k=4$. Let $u$ be the vertex having label $1$, and partition the vertex set of $\G$ into the sets $S_0,\,S_1,\,S_2,\,S_3$, where $v \in S_i$ if and only if $d(u,v) = i$ (in particular, $S_0 = \{u\}$). Clearly, all the vertices in $S_1$ are dominated by $u$. We split the proof into two cases:

\textbf{Case 1:} $u$ is dominated by a vertex $v$ in $S_2$ having label $2$.\\
By vertex-transitivity, $v$ dominates $|S_2|$ vertices, hence, since $\Gamma$ is bipartite, by a counting reasoning it dominates $S_0 \cup S_2 \setminus \{v\}$. Clearly $v$ must be dominated by the vertex having label $3$, and since every other vertex in $S_0 \cup S_2 \setminus \{v\}$ is already covered, by vertex-transitivity we have $|S_3| = 1$. Call then $w$ the vertex having label $3$. Clearly, $w \not \in S_3$, otherwise it would dominate $u$, hence $w \in S_1$. Since $\mathrm{diam}(\Gamma) = 3$, every vertex in $S_2$ has either distance $1$ or $3$ from $w$, and from $|S_3| = 1$ and vertex-transitivity we deduce that $v$ is the unique vertex having distance $3$ from $w$. Thus, $w$ is adjacent to every vertex of $S_2 \setminus \{v\}$. We then have $\mathrm{deg}(w) = |S_0|+ |S_2\setminus\{v\}| = |S_2|$, and since $\Gamma$ is regular and bipartite $\mathrm{deg}(w) = |S_2| = |S_1| = \mathrm{deg}(u)$. It then follows that $\Gamma$ is a regular bipartite graph on $2(|S_1|+1)$ vertices, having degree $|S_1|$, hence $\Gamma$ is a crown graph.

\textbf{Case 2:} $u$ is dominated by a vertex $v$ in $S_3$ having label $3$.\\
Since $\Gamma$ is bipartite, $v$ dominates $u$ and $|S_3|-1$ vertices of $S_2$, hence to dominate the vertices in $S_3$ (in particular $v$) we use the remaining labels $0$ and $2$. Let $w$ be the vertex having label $2$.
\begin{itemize}
\item If $w \in S_1$, then $w$ must be adjacent to every vertex in $S_1$, otherwise there would be a vertex that is dominated by both $u$ and $w$. If $|S_1|>1$, the graph $\Gamma$ would not be bipartite, while if $|S_1| = 1$, that is if $S_1 = \{w\}$, then $\Gamma$ would be a regular graph of degree $1$, that is the path of length $1$, and $\mathrm{diam}(\Gamma) = 1 \neq 3$. In any case, we reach a contradiction.
	\item If $w \in S_3$, then there exists a path $P$ realizing the minimum distance between $u$ and $w$, namely $P = [u, s_1,s_2 ,w]$, with $s_1\in S_1$ and $s_2\in S_2$. We have then that $u$ and  $w$ dominate $s_1$, hence by Theorem \ref{prop:disjoint} it is not an extended irregular dominating labeling.
\end{itemize}
Hence, Case 2 cannot occur.
\end{proof}

In the remaining part of this section we present a complete solution for the existence problem of an optimal extended irregular dominating set in the case of hypercubes and Mobius ladders.

We recall that given a positive integer $n$, the \textit{hypercube} of dimension $n$, that we denote here by $\Q_n$, is the graph whose vertex set is identified by the sequences in $\{0,1\}^n$, and whose edges connect vertices having Hamming distance equal to $1$. Clearly, $\Q_n$ is a bipartite graph with diameter $n$.

\begin{theorem}
The $n$-dimensional hypercube admits an optimal extended irregular dominating set if and only if $n=0,3$. 
\end{theorem}
\begin{proof}
The existence of an extended irregular dominating set is trivial for $\Q_0$, which consists of a single vertex, and Example \ref{ex:ext_dom_set} shows the case $\Q_3$. Also, the non existence for $\Q_1$ and $\Q_2$ follows from Corollary \ref{cor:completo}.

Consider now $\Q_4$ shown in Figure \ref{Q4}. It is not restrictive to assume that the vertex $v=(0,0,0,0)$ has label $2$. A direct check shows that $v$ cannot be covered by a vertex having labels $1$ or $3$, otherwise there should be some vertex covered twice. Hence it must be covered by a vertex having label $4$, that has to be $w=(1,1,1,1)$. It is then easy to see that this cannot be completed to an extended irregular dominating set. 
\begin{figure}
\begin{center}
\includegraphics[width = 0.5\textwidth]{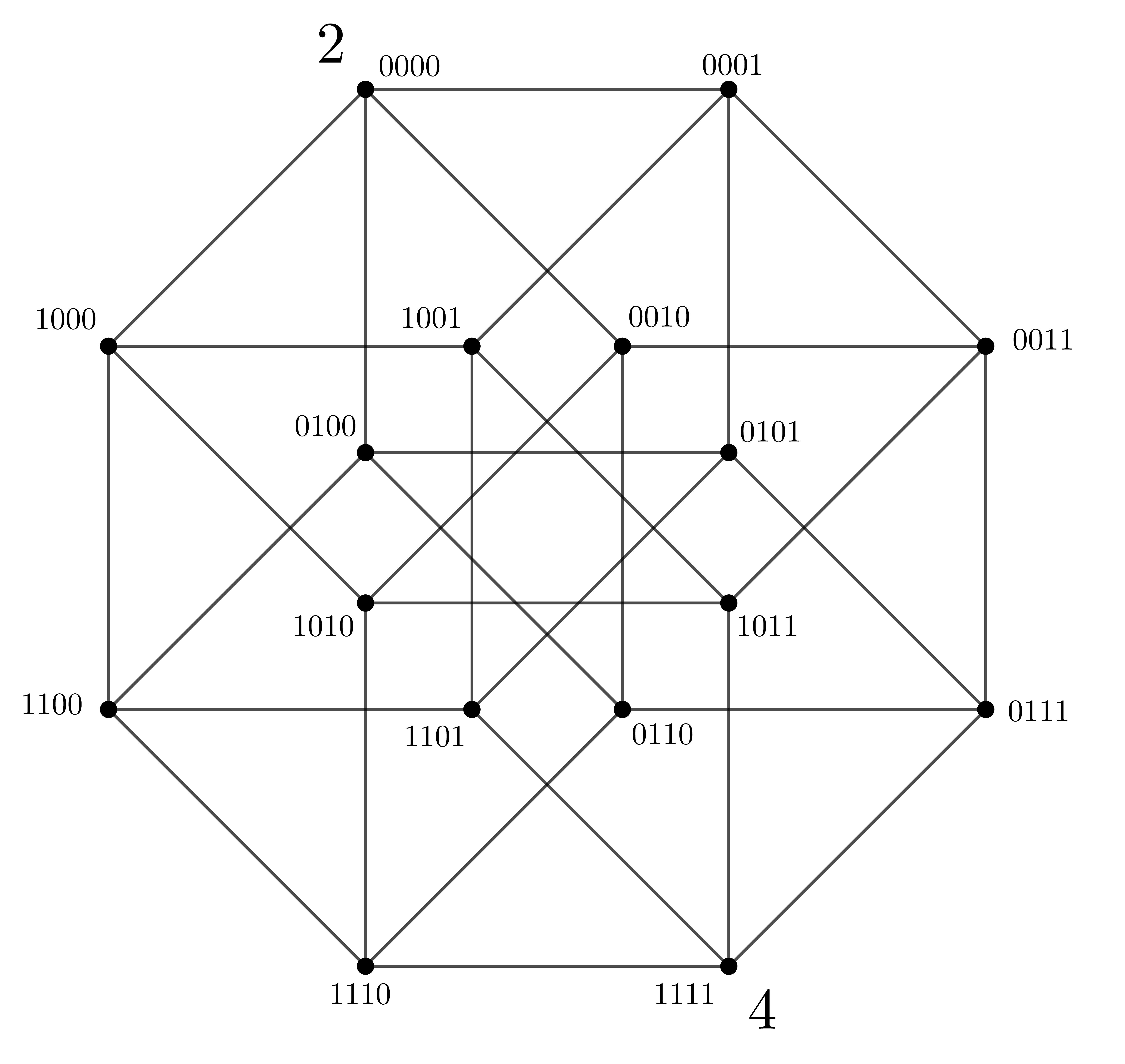}
\caption{The hypercube $\Q_4$.}\label{Q4}
\end{center}
\end{figure}

Assume now  that there exists an irregular dominating set in $\Q_n$ for some integer $n > 4$.
For any fixed vertex $v$ of $\Q_n$, and for every $d\in[0,n]$ let $S_d$ denote the set of vertices in $\Q_n$ having distance $d$ from $v$, where it is understood that $S_0=\{v\}$. It is  easy to see that $|S_d| = \binom{n}{d}$, and that if $u \in S_d$ for some $d \in [1,n]$, then $N_{\Q_n}(u) \subseteq S_{d-1} \cup S_{d+1}$ where $N_{\Q_n}(u)$ denotes the set of vertices of $\Q_n$ adjacent to $u$.

 Without loss of generality assume that  $v = (0,0,\dotsc,0)$ is the vertex having label $4$. We then have that $\Q_n \setminus S_4$ has two connected components, and in particular the one containing $v$, say $C$, has $1+n + \binom{n}{2}+ \binom{n}{3}$ vertices. If $A$ is the part of the bipartition of $\Q_n$ containing $v$, then $|A \cap C| = 1+ \binom{n}{2}$.  

Assume that the vertex $v$ is dominated by a vertex $w = (w_1,\dotsc, w_n)$ having label $d$: this implies that $w \in S_d$, thus there is a $d$-set $I = \{i_1,\dotsc, i_d\}$ such that $w_{i}=1$ if and only if $i \in I$. Now, if $2 \leq d \leq n-2$, let $i_1,i_2,j_1,j_2$ be four distinct indexes, with $i_1,i_2 \in I$ and $j_1,j_2 \not \in I$; let $z = (z_1,\dotsc,z_n)$ be the vertex having coordinates:
\[
z_i = \left\{
\begin{aligned}
&1 & \text{ if $i \in \{i_1,i_2,j_1,j_2\}$},\\
&0 & \text{otherwise}.
\end{aligned}
\right.
\] 
It is easy to see that $z \in S_4$, and that the Hamming distance between $z$ and $w$ is $d$, hence $z$ is dominated by $v$ and $w$, and the labeling does not induce an extended dominating set. 

We now show that it is not possible to dominate $v$ using the labels $d= 1$ or $d= n-1$. Let then $w$ in $S_1$ be the vertex that dominates $v$, having label $1$ (it is easy to see that the case where the label is $n-1$ is analogous to this one). If, without loss of generality, $w = (w_1,w_2,\dots, w_n) = (1,0,\dotsc, 0)$, then $w$ covers the $n-1$ vertices of $S_2$ having the first coordinate equal to $1$. We show that it is not possible to cover the remaining vertices of $S_2$, with a vertex $u$ having label $\ell$. As a first remark, the label $\ell = 2$ cannot be used to cover the remaining vertices of $S_2$; indeed, it is easy to see that this implies $u = (u_1,\dotsc, u_n) \in S_4$, hence,
since $n >4$, then there exists at least one component $u_i = 0$. If $u_j =1$ for some $j \in [1,n]$, then the vertex $z = (z_1,\dotsc, z_n)$ having $z_m = u_m$ for each $m \in [1,n] \setminus \{i,j\}$, $z_i = 1$ and $z_j = 0$ has distance $2$ from $u$, but belongs to $S_4$, thus it is dominated twice.

Let then $u$ be labeled with $\ell >2$. Clearly, $u \not \in S_\ell$, otherwise it would dominate $v$ another time,  and if $u \in S_m$, with $m \leq \ell -4$ or $m \geq \ell+4$, then $u$ does not dominate any vertex of $S_2$. Then:
\begin{itemize}
\item if $u \in S_{\ell-2}$, then there are precisely $\ell -2\geq 1$ indexes $i_1,\dotsc, i_{\ell-2}$ such that $u_{i_j} =1$ for $j \in [1,\ell-2]$, and $0$ otherwise. In particular, if $n-1 > \ell$, then there are at least three zero entries $u_a, \,u_b, \,u_c$ (indeed, two of them are required to dominate vertices of $S_2$, while the third one is ensured by $n-1 > \ell$). Let then $z = (z_1,\dotsc, z_n)$ be the vertex such that $z_m = 1$ if and only if $m \in \{a,b,c,i_1\}$: $z$ belongs to $S_4$ and has distance $\ell$ from $u$, thus it is dominated twice. Now, if $\ell = n$, then $u$ would dominate precisely one vertex of $S_2$, but the vertices of $S_2$ that are not dominated by $w$ are:
\[
\binom{n}{2} - (n-1) = \frac{n(n-3)}{2} +1.
\]
It would be then necessary to use other labels to cover the remaining vertices of $S_2$, thus returning in one of the other cases.
\item if $u \in S_{\ell+2}$, then there are precisely $\ell +2\geq 3$ indexes $i_1,\dotsc, i_{\ell+2}$ such that $u_{i_j} =1$ for $j \in [1,\ell-2]$, and $0$ otherwise. If $\ell \leq n-3$, then there is at least one zero entry  $u_a$: if $z= (z_1,\dotsc, z_n)$ is the vertex such that $z_m = 1$ if and only if $m \in \{a,i_1,i_2,i_3\}$, then $z \in S_4$ and is dominated twice. If $\ell = n-2$, necessarily $u \in S_n$, but then $u$ would dominate the whole set $S_2$, that is already partially covered by $w$.
\end{itemize}

To conclude, assume that $v$ is dominated by the vertex $w = (1,\dotsc,1) \in S_n$ having label $n$.  Assume now that some of the vertices of $S_2$ are dominated by a vertex $u = (u_1,\dotsc, u_n)$ having label $\ell$. Clearly, $\ell \neq 1$, otherwise $u$ would cover either $v$ or vertices of $S_4$. If $\ell = 2$, since $n >4$ there are vertices in $S_4$ that are dominated by $u$ (see above). 
If $\ell > 2$, then the reasoning explained above can be applied for almost all the cases: the only exceptions are  for $\ell = n-2$, that in this case would imply $u = w$, and for $\ell = n$, that it is not possible as that label is already used.

It then follows that it is not possible to cover the vertices in $S_0 \cup S_2$ with a labeling that induces an extended irregular dominating set, hence the statement follows.
\end{proof}

We recall now the definition of a Mobius ladder.
For every positive integer $n\geq2$, the \emph{Mobius ladder} on $2n$ vertices, denoted by $M_{2n}$, is the graph having $\{x_1,x_2,\dotsc, x_{2n}\}$ as vertex set, and such that $E(M_{2n})=\{\{x_i,x_{i+1}\} \mid 1\leq i\leq 2n\}\cup \{\{x_i, x_{i+n}\} \mid 1\leq i \leq n\}$, where the subscripts have to be considered modulo $2n$. Note that $M_4$ is nothing but the complete graph on $4$ vertices.
\begin{lemma} \label{lem:mobius_2_distance}
Let $n \geq 3$ and let $x$ be a vertex of $M_{2n}$. Assign a label $\ell \in [1,\lceil\frac{n}{2}\rceil]$ to $x$. Then, for every vertex $v$ covered by $x$ there exists a vertex $u$, dominated by $x$, such that $d(u,v) = 2$.
\end{lemma}
\begin{proof}
Clearly, by vertex-transitivity, we can assume that $x = x_1$.
The vertices dominated by $x_1$ are $x_{1+\ell}$, $x_{2n-\ell+1}$, $x_{n+\ell}$ and $x_{n-\ell+2}$ (note that for $n$ odd and $\ell =\frac{n+1}{2}$ we have $x_{1+\ell} = x_{n-\ell+2} = x_{\frac{n+3}{2}}$ and $x_{2n-\ell+1} = x_{n+\ell} = x_\frac{3n+1}{2}$). Since $\{x_i, x_{i+n}\}$ is an edge of $M_{2n}$ for every $i \in [1,n]$, we have $d(x_{1+\ell}, x_{n+\ell}) = d(x_{n-\ell+2}, x_{2n-\ell+1}) =2$.
\end{proof}

\begin{theorem}
    For every $n\geq 2$, the Mobius ladder $M_{2n}$ does not admit an extended irregular dominating set. 
\end{theorem}
\begin{proof}
Since $M_4$ is the complete graph of order $4$ and $M_6$ is the complete bipartite graph $K_{3,3}$, for $n=2,3$, the result follows from Corollary \ref{cor:completo}.

Suppose now $n\geq 4$. Assume by contradiction that there exists an extended irregular dominating set of $M_{2n}$. By vertex-transitivity, we can suppose, without loss of generality, that $x_1$ is the vertex that receives label $2$, hence dominating $X= \{x_3, x_{n},x_{n+2},x_{2n-1}\}$.

If $n$ is even, from Lemma \ref{lem:mobius_2_distance}, it can be seen that it is not possible to cover $x_1$ with a vertex having label in $[1,\frac{n}{2}-1]$: indeed, there would be a vertex between $x_3,\, x_{n},\,x_{n+2},\,x_{n-1}$ covered twice. Hence, $x_1$ has to be covered with a vertex having label $\frac{n}{2}$. It can be seen that the graph induced by the vertices dominated by a vertex having label $n/2$ is a path $P$ on $4$ vertices. However, $M_{2n} \setminus X$ is a disconnected graph, where the connected component containing $x_1$ is isomorphic to the complete bipartite graph $K_{1,3}$: since $P$ is not a subgraph of $K_{1,3}$, it follows that $x_1$ cannot be covered. Hence an extended irregular dominating set of $M_{2n}$ does not exist.

Suppose now $n$ odd.  By Lemma \ref{lem:mobius_2_distance} it can be seen that if a label in $\ell \in[1,\frac{n+1}{2}]$ is assigned to a vertex $w$, and $d(w,x_1) =\ell$, then there exists a vertex  $x\in X$ such that $d(w,x) = \ell$, that is then covered twice. Hence, there does not exist an extended irregular dominating set.
\end{proof}

\section{Results for cycles via strong starters}\label{sec:cycles}
In this section, we show that optimal extended irregular dominating labelings of cycles are equivalent to a combinatorial structure that has been thoroughly studied over the course of the years, that is strong starters in a cyclic group, see \cite{StarterHB}.

\begin{definition}\label{def:starter}
    Let $G$ be an additive abelian group of odd order $g$, where the neutral element is denoted by $0$. A \textit{starter} $L$ in $G$ is a set of unordered pairs $\{\{x_i,y_i\} \mid 1 \leq i \leq (g-1)/2\}$ such that:
\begin{itemize}
	\item[$\mathrm{(1)}$] $\{x_i,y_i  \mid 1 \leq i \leq (g-1)/2 \} = G \setminus \{0\}$;
	\item[$\mathrm{(2)}$] $\{ \pm (x_i-y_i) \mid 1 \leq i \leq (g-1)/2\} = G \setminus \{0\}$.
\end{itemize} 
A  starter $L=\{\{x_i,y_i\}\}$ in   $G$ is called a \textit{strong starter} if the additional property:
\begin{itemize}
\item[$\mathrm{(3)}$]
$x_i+y_i=x_j+y_j$ implies $i=j$ and for any $i$, $x_i+y_i\neq 0$
\end{itemize}
is satisfied. In other words a starter is called strong if $\{ (x_i+y_i) \mid 1 \leq i \leq (g-1)/2\}$ comprises of distinct elements in $G\setminus \{0\}$.

A  starter $L=\{\{x_i,y_i\}\}$ in   $G$ is said to be \textit{skew} if the following additional property holds:
\begin{itemize}
\item[$\mathrm{(4)}$] $x_i+y_i=\pm(x_j+y_j)$ implies $i=j$ and for any $i$, $x_i+y_i\neq 0$.
\end{itemize}
or equivalently if:
\begin{itemize}
\item[$\mathrm{(4)}$] 
$\{ \pm (x_i+y_i) \mid 1 \leq i \leq (g-1)/2\} = G \setminus \{0\}$.
\end{itemize}
\end{definition}

It is clear that a skew starter is also a strong starter. Both strong and skew starters have been studied in many groups, achieving various existence results. Here, we are interested in the case of the cyclic group $\Z_n$ for some positive integer $n$.
An hill-climbing algorithm to find strong starters in cyclic groups has been developed in \cite{DS1}. 
In \cite{Stinson24}, Stinson presented several results for strong starters in view of which he proposed the following conjecture.
\begin{conjecture}\label{conj:Stinson}
Let  $n\geq 5$ be an odd integer. There exists a strong starter in $\Z_n$ if and only if $n\neq5,9$.
\end{conjecture}
It is easy to see that a starter of $\Z_5$ does not exist. Also, the non existence of a strong starter
of $\Z_9$ is well-known. In the same paper Stinson proved the existence of a strong starter in $\Z_n$
for every odd $n\neq 9$, with $7\leq n\leq 99$.
We point out that the conjecture proposed in \cite{Stinson24} has a more general statement that Conjecture 
\ref{conj:Stinson}, but for the purpose of this paper it is sufficient to focus on this special case.
We also underline that this special case is contained in the following conjecture \cite{H} proposed by Horton in 1990, which is not restricted to cyclic groups and which is still far from being solved.

\begin{conjecture}
  Suppose that $G$ is an abelian group of odd order $g\geq3$. Then there is a strong starter in $G$ if and only if $G\neq  \Z_3, \Z_5, \Z_9$ or $\Z_3 \times \Z_3$.  
\end{conjecture}

Now we summarize the main results for skew starters in a cyclic group, obtained in \cite{CGZ,L,LS,MN}.
\begin{theorem}\label{thm:skew_starters}
Let $n$ be a positive integer. Then, there exists a skew starter in $\Z_n$ in the following cases:
\begin{itemize}
	\item $n = 2^kt+1$ is a prime power, where $t>1$ is an odd integer (\textit{Mullin-Nemeth starters});
	\item  $n = 16t^2+1$ (\textit{Chong-Chang-Dinitz starters});
 \item $\gcd(n,6) =1$ and either $n \neq 0 \pmod{5}$ or $n = 0 \pmod{125}$.
\end{itemize}
\end{theorem}
As remarked by Stinson in \cite{Stinson}, the parameters for the Mullin-Nemeth and Chong-Chang-Dinitz starters allow to construct skew starters in the notable class of cyclic groups whose order is a prime number larger than $5$.

We now establish the following equivalence, where by $C_n$ we denote the cycle of length $n$.
\begin{proposition}
Let $n$ be an odd integer. Then, an optimal extended irregular dominating set over $C_n$ is equivalent to a strong starter in $\Z_n$.
\end{proposition}

\begin{proof}
Let $S$ be an extended irregular dominating set over $C_n = (v_0,v_1,\dotsc, v_{n-1})$, with labeling function $\lambda: S \rightarrow [0,\frac{n-1}{2}]$. Recall that, by Corollary \ref{cor:k=radius}, $|S|=\frac{n+1}{2}$, that is, $\lambda$ is a bijection. Observe that, since $n$ is odd, $2$ admits a multiplicative inverse in $\Z_n$, that we denote by $2^{-1}$ for the sake of brevity.
Assume without loss of generality that $v_0\in S$ and $\lambda(v_0) =0$, and construct  the following set $L \subset \Z_n \times \Z_n$:
\[
L = \{ \{i,j\} \mid \text{ $v_i,v_j$ are dominated by $v \in S\setminus \{v_0\}$ } \}.
\]
We prove that $L$ is a strong starter in $\Z_n$.  Since every vertex of $C_n$ is dominated exactly once, it is immediate to verify that property $\mathrm{(1)}$ of Definition \ref{def:starter} holds. 
Moreover, let $\{i,j\} \in L$, with $i >j$, and let $v_k$ be the vertex that dominates $v_i$ and $v_j$, receiving label $\ell$. It can be immediately seen that precisely one between $i-j$ and $n+j-i$ is even. 
In the first case, we necessarily have $v_k = v_{(i+j)/2}$ and $\ell = \frac{i-j}{2}$, while in the second case $v_k = v_{(n+i+j)/2}$ and $\ell = \frac{n+j-i}{2}$.
Suppose now that property $\mathrm{(2)}$ of Definition \ref{def:starter} does not hold, and let $\{x,y\}$ and $\{r,s\}$ be two pairs such that $\{\pm (x-y)\} = \{\pm (r-s)\}$, with $x>y$ and $r>s$. Note that it is not restrictive to assume that $x-y = r-s$ is an even number (otherwise, consider $y+n$ and $s+n$). Then, the labels assigned to the vertices $v_{\frac{x+y}{2}}$ and $v_{\frac{r+s}{2}}$ are not distinct, hence $\lambda$ is not a bijection and $S$ is not an extended irregular dominating set, that is a contradiction. Thus property $\mathrm{(2)}$ of Definition \ref{def:starter} holds. Hence we have proved that $L$ is a starter. Finally, it is easy to see that if property $\mathrm{(3)}$ of Definition \ref{def:starter} does not hold,  there exist $\{x,y\}$ and $\{r,s\}$ in $L$ such that $x+y = r+s$, then the vertices dominating $v_x,v_y$ and $v_r,v_s$ must coincide. It would follow that $S$ is not an extended irregular dominating set, hence also property $\mathrm{(3)}$ of Definition \ref{def:starter} holds, and $L$ is a strong starter in $\Z_n$.

Let now $L = \{\{x_i,y_i\} \mid 1 \leq i \leq (n-1)/2\} $ be a strong starter in $\Z_n$ with $x_i>y_i$ for each $1 \leq i \leq (n-1)/2$. Let $V=\{v_0,v_1,\ldots,v_{n-1}\}$ be the vertex set of $C_{n}$ and $E(C_n)=\{\{v_i,v_{i+1}\} \mid i\in[0,n-1]\}$ where the indexes are understood modulo $n$.
Consider the following set  $\left\{v_{\frac{x_i+y_i}{2}} \mid \{x_i,y_i\} \in S\right\}\subseteq V$, where if $x_i+y_i<n$ is odd, by $v_{\frac{x_i+y_i}{2}}$ we mean $v_{\frac{n+x_i+y_i}{2}}$. 
Take now the labeling so defined:
\[
\lambda\left(v_{\frac{x_i+y_i}{2}}\right) = \left\{ \begin{aligned}
&\frac{x_i-y_i}{2} \quad \text{ if $x_i - y_i\pmod{n}$ is even,} \\
&\frac{n-(x_i-y_i)}{2} \quad \text{otherwise.} \\ 
\end{aligned}\right.
\]
One can easily check that $\lambda$ is an optimal extended irregular dominating labeling of $C_n$.
\end{proof}

Since every skew starter is a strong starter, the existence of an optimal extended irregular dominating set is granted for cycles length $n$, for every $n$ belonging to one of the cases of Theorem \ref{thm:skew_starters} and for every $n$ for which Conjecture \ref{conj:Stinson} holds. In particular, we have.

\begin{corollary}\label{cor:Cnodd}
The cycle $C_n$ admits an optimal extended irregular dominating set for every odd integer $n\neq 9$, such that $7\leq n \leq 99$ and for every prime $n>5$.
\end{corollary}

\begin{example}
    Starting from the following strong starter of $\mathbb{Z}_{17}$
    $$\{\{9,10\},\{3,5\},\{13,16\},\{11,15\},\{1,6\},\{2,8\},\{7,14\},\{4,12\}\}$$
    one can construct the optimal extended irregular dominating labeling of $C_{17}$ below.
    \begin{center}
        \includegraphics[width= \textwidth]{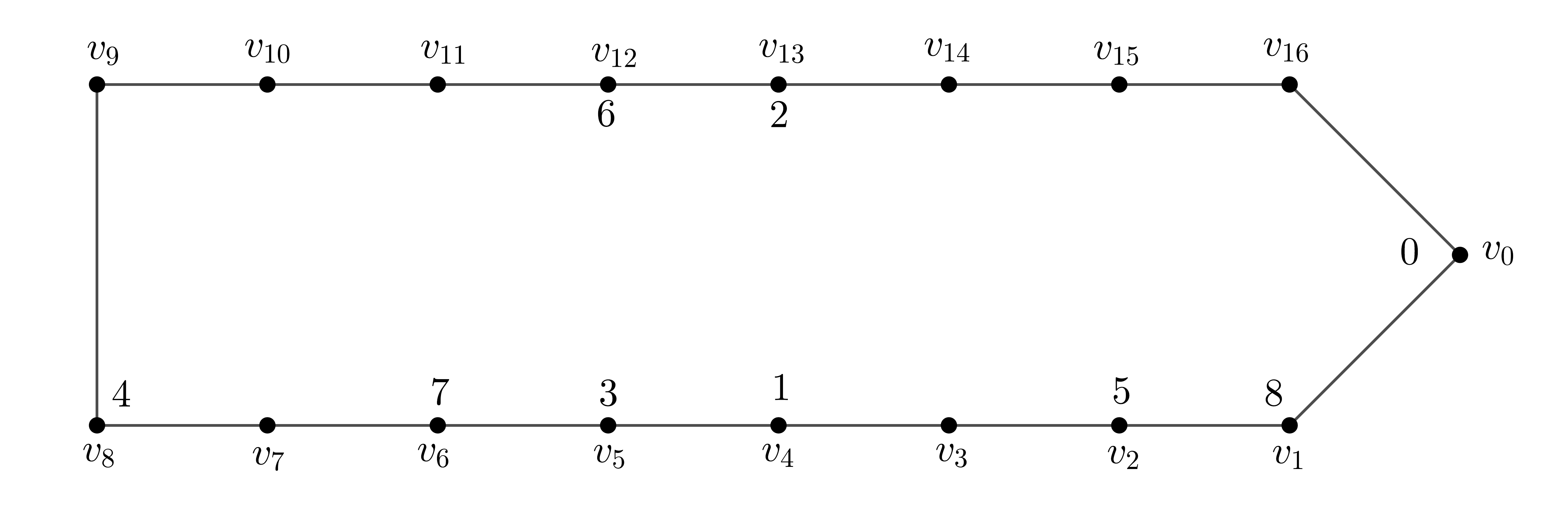}
    \end{center}
\end{example}

In what follows, we show that there exists an optimal extended irregular dominating set for many cycles, having singly even length.

\begin{proposition} \label{prop:doubling}
Let $n$ be an odd integer. If  $C_n$ admits an optimal extended irregular dominating set, then $C_{2n}$ admits an optimal extended irregular dominating set too.
\end{proposition}
\begin{proof}
Since $n$ is an odd integer, for every $x \in \Z_{n}^*$ precisely one element between $x$ and $n-x$ is odd. Let $C_{2n} = (v_0,v_1,\dotsc, v_{2n-1})$, and for any vertex $v_i$ by $-v_i$ we mean $v_{i+n}$, where indexes are read modulo $2n$. 
Let now $\lambda$ and $\lambda' $ be two labelings on $C_{2n}$ such that $\lambda(v) = x$ and $\lambda'(-v) =n -x$ for some vertex $v$, where $x,n-x \in \Z_{n}^*$. Then $v$ through $\lambda$ and $-v$ through $\lambda'$ dominate the same set of vertices. Set $A = \{v_{2i}: i \in [0,n-1]\}$ and $B = \{v_{2i+1}: i \in [0,n-1]\}$. 

Let $V(C_n)= (x_0,\dotsc, x_{n-1})$, and set $\lambda_1: X\subset V(C_n) \rightarrow [0,\frac{n-1}{2}]$ be an optimal extended irregular dominating labeling of $C_n$ which exists by hypothesis. Let $\psi$ be the natural bijection $\psi: V(C_n) \rightarrow A \subset V(C_{2n})$, where $x_i \mapsto \psi(x_i) = v_{2i}$. Construct the labeling $\lambda_2 : \psi(X) \subset V(C_{2n}) \rightarrow [0,n]$, where $\lambda_2(v_{2i}) = \lambda_2\psi(x_i) = 2 \lambda_1(x_i) $. Since the vertices in $\psi(X)$ are contained in $A$ and have even labels, from the fact that $\lambda_1$  induces an optimal extended irregular dominating set  it follows that the vertices of $\psi(X)$ dominate $A$.

Let now $v_{2i} \in A$ be any vertex that is not labeled by $\lambda_2$, that is $v_{2i} \in A \setminus \psi(X)$. Let $Y = \{ v_{2i},-v_{2i} \}$ be a cut of $C_{2n}$, and let $\C_1$ and $\C_2$ be the vertex set of the two connected components of $C_{2n} \setminus Y$. For the sake of brevity, denote by $Z$ the set $(\psi(X) \cap \C_1) \cup (-(\psi(X) \cap \C_2))$, and let $\lambda_3: Z \subset \C_1 \rightarrow [0,n]$ be the following labeling:
\[
\lambda_3(v)=\left\{
\begin{aligned}
&\lambda_2(v)& \text{ if $v \in \psi(X) \cap \C_1$,} \\
&-\lambda_2(v) \pmod{n}& \text{ if $\lambda_2(v) \neq 0$ and  $v \in-(\psi(X) \cap \C_2) $,} \\
&n &\text{if  $\lambda_2(v) = 0$ (and $v \in-(\psi(X) \cap \C_2) $).} \\
\end{aligned}\right.
 \]
As previously remarked, it follows that $\lambda_3$ induces a set of labeled vertices dominating $A$. Let now $\eta: \C_1 \rightarrow \C_2$, $v \mapsto \eta(v) = w$, where $w$ is the vertex of $\C_2$ such that the distance between $w$ and $v_{2i}$ is equal to the distance between $v$ and $v_{2i}$.
We conclude the proof by constructing the labeling $\lambda: Z \cup \eta(Z) \rightarrow [0,n]$, where:
\[
\lambda(v)=\left\{
\begin{aligned}
&\lambda_3(v) &\text{ if $v\in Z$,} \\
&-\lambda_3(v)& \text{if $\lambda_3(v) \neq n$ and $v \in -Z$,}\\
&0& \text{if $\lambda_3(v) = n$ (and $v \in -Z$).}
\end{aligned}\right.
\]
Since $Z \subset \C_1$, it follows that $\lambda$ is well-defined, and as the dominated set induced by $\lambda_3$ is $A$, $\lambda$ is an optimal extended irregular dominating labeling inducing an optimal extended irregular dominating set on $C_{2n}$. 
\end{proof}

It is easy to see that the converse of previous result does not hold. For example, even if $C_3$ and $C_5$ do not admit an extended irregular dominating set, $C_6$ and $C_{10}$ have such a set as shown below:
\[
\begin{aligned}
 &C_6:&(3,1,\square,2,0,\square), \qquad
&C_{10}: &(1,4,2,\square,\square,3,0,\square,5,\square), \\
\end{aligned}
\] 
where $\square$ denotes a vertex with no label.

\begin{corollary}
Let $p$ be a prime with $p\geq 7$. Then, there exists an optimal extended irregular dominating set of $C_{2p}$.
\end{corollary}

\begin{corollary}
Let $n$ be an odd integer with $n\neq 9$ and $7\leq n \leq 99$. Then, there exists an optimal extended irregular dominating set of $C_{2n}$.
\end{corollary}

\begin{example}\label{ex:C14}
We show the construction provided in Proposition \ref{prop:doubling} for $n=7$. Consider the following optimal extended irregular dominating labeling of $C_7$:
\begin{center}
\includegraphics[width = 0.5 \textwidth ]{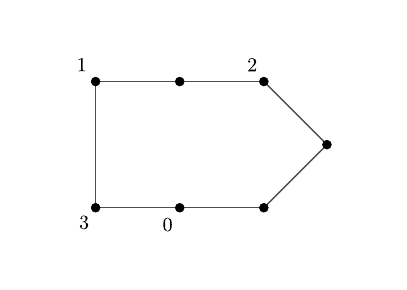}
\end{center}
We now consider the cycle $C_{14} = (v_0,v_1,\dotsc,v_{13})$, and a possible choice for the labeling obtained by doubling the labels of $C_7$: 
\begin{center}
\includegraphics[width = \textwidth]{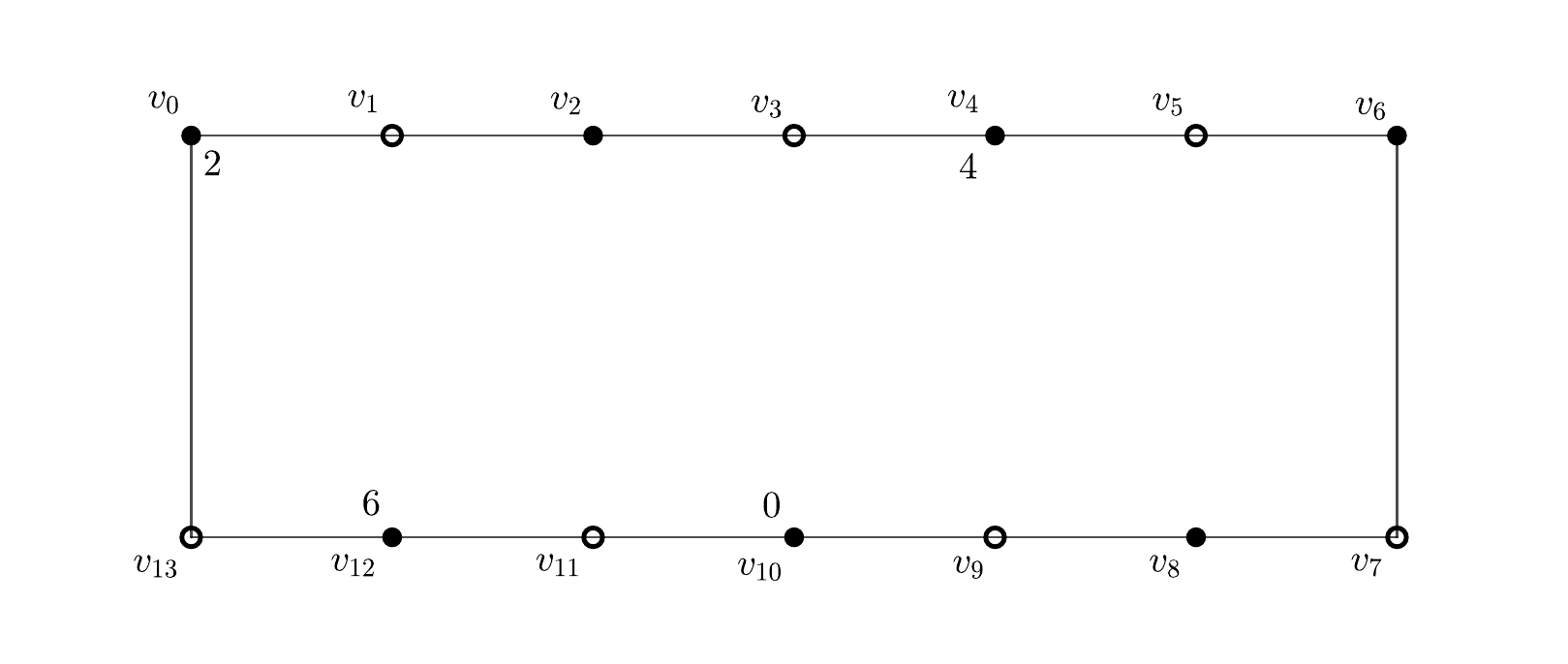}
\end{center}
It can be seen that the labeled vertices cover all the vertices of the form $\{v_{2i}: i \in[0,6]\}$. Consider now the cut $Y=\{v_1,v_8\}$, and the vertex sets  $\C_1 = \{v_i: i \in [2,7]\}$ and $\C_{2} = \{v_i: i \in \{0\}\cup[9,13]\}$ of the connected components of $C_{14}\setminus Y$. We then have $Z = \{v_4\} \cup \{v_3,v_5,v_7\} \subset \C_1$, and we construct the labeling $\lambda_3: Z \rightarrow [0,7]$:
\[
\begin{aligned}
&\lambda_3(v_4) = 4, & \lambda_3(v_3) = 7, \\
& \lambda_3(v_5)  =  1, &\lambda_3(v_7) =5, \\  
\end{aligned}
\]
that we show below:
\begin{center}
 \includegraphics[width =0.9 \textwidth]{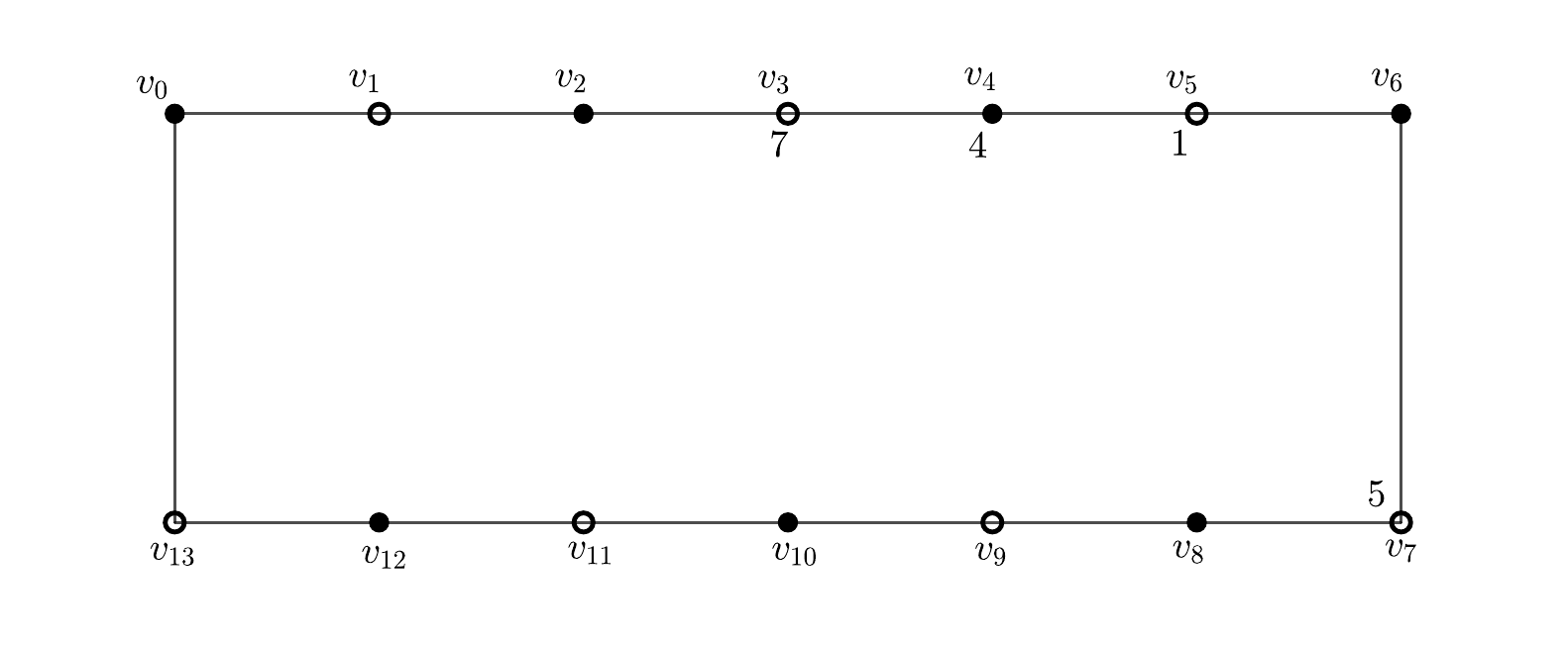}
\end{center}
We then conclude the example by showing the labeling $\lambda$. Since $\lambda$ can also be obtained by reflecting the labeled vertices along the line joining the vertices of the cut $Y$, we have added the dashed line between $v_1$ and $v_8$:
\begin{center}
\includegraphics[width = 0.9\textwidth]{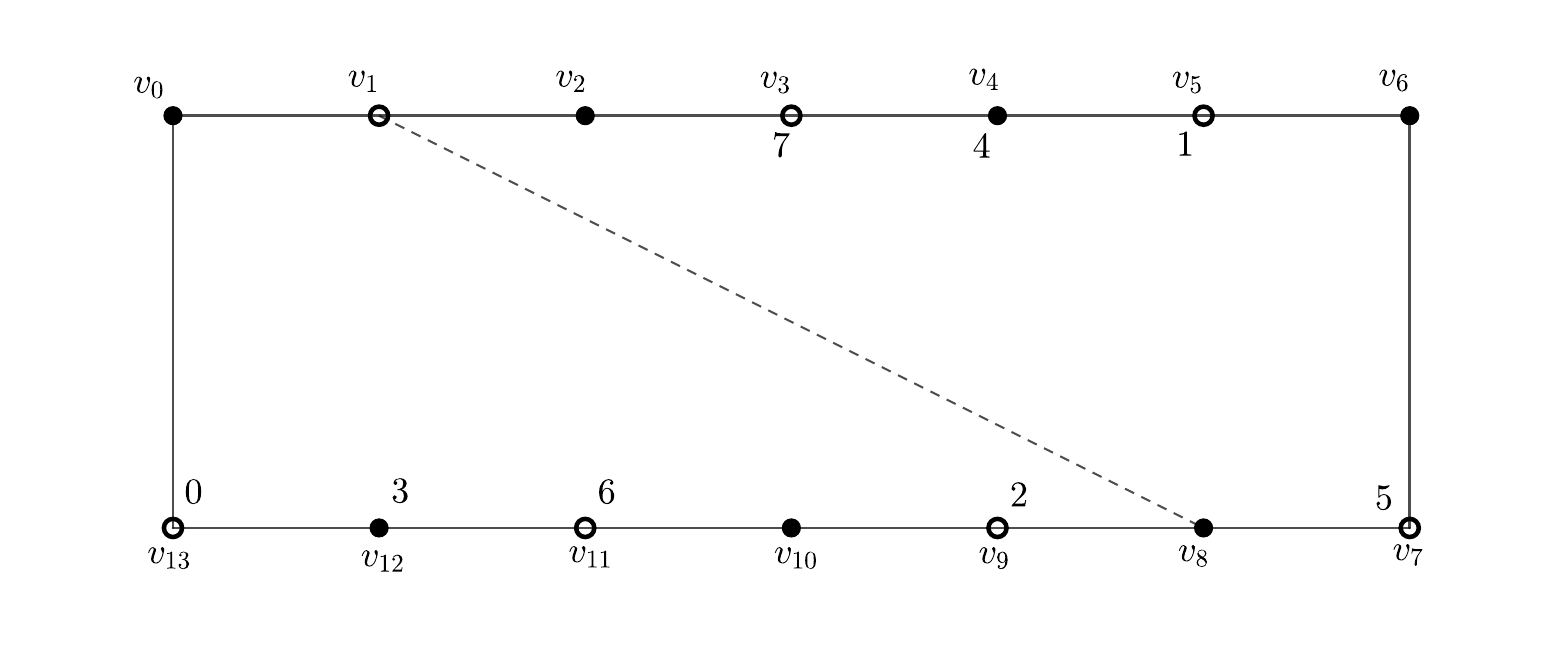}
\end{center}
It can then be seen that the labeling $\lambda$ so defined induces an optimal extended irregular dominating set.
\end{example}

We then summarize these results:

\begin{corollary}
Let $n$ be a positive integer, where either
\begin{itemize}
\item $n\geq 7$ is a prime,
\item $n$ is an odd integer with $n\neq 9$ and $7\leq n \leq 99$,
	\item $n = 2^kt+1$ is a prime power, where $t>1$ is an odd integer,
	\item $n = 16t^2+1$, or
 	\item $\gcd(n,6) =1$ and either $n \neq 0 \pmod{5}$ or $n = 0 \pmod{125}$.
\end{itemize}
Then, there exists an optimal extended irregular dominating set of $C_n$ and $C_{2n}$.
\end{corollary}

\section{A class of non vertex-transitive graphs: the paths}\label{sec:paths}
Obviously it makes sense to consider the irregular extended domination problem also for non vertex-transitive graphs.  In this section we focus on paths.  
By $P_n$ we denote the path on $n$ vertices, that we also write as its list of vertices $[x_1,x_2,\dotsc,x_n]$, where the edges are $\{x_i,x_{i+1}\}$ for every $i \in [1,n-1]$.
First of all we completely solve the existence problem of an extended irregular dominating set for a path, then we establish   $\gamma_e(P_n)$ for several values of $n$. To do this, let us see some connection between irregular dominating sets and extended irregular dominating sets for paths.

\begin{remark}\label{rem:removek}
Let $S$ be a $k$-irregular dominating set of $P_n$. If there exists a vertex $v$ dominating a unique vertex $u$ of $\G$ and $u\not\in S$, then  there also exists a $k$-extended irregular dominating set of $P_n$. In fact it is sufficient to remove the label from $v$ and to label $u$ by $0$.
\end{remark}

\begin{remark}\label{rem:add0}
    If there exists a $k$-irregular dominating labeling, say $\lambda$, of the path $P_n=[x_1,x_2,\ldots,x_n]$, then there exists a $(k+1)$-extended irregular dominating labeling, say $\lambda'$, of $P_{n+1}=[x_1,x_2,\ldots,x_n,x_{n+1}]$. In fact it is sufficient to extend $\lambda$ by labeling the vertex $x_{n+1}$ by $0$. 
      Unfortunately, if $\lambda$ is optimal, this does not necessarily imply that $\lambda'$ is optimal too.
\end{remark}
Clearly, the same reasoning can be applied also to other classes of graphs.

\begin{example}
In \cite{BCZ21}  the following $6$-irregular dominating labeling of $P_8$ is presented $$ [\square ,5,3,1,4,2,\square, 6],$$
where $\square$ denotes a vertex with no label, and it is proved that it is optimal, that is $\gamma_i(P_8)=6$.
Such a labeling can be extended to the following $7$-extended irregular dominating labeling of $P_9$: $$\quad [\square ,5,3,1,4,2,\square, 6,0]$$
which is not optimal in fact there exists a  $6$-extended irregular dominating labeling of $ P_9$: $$ \quad[\square, \square ,2, \square, 3, 0, 4, 1, 5 ].$$ 

\end{example}

The existence problem for an irregular dominating set of a path has been completely solved in \cite{BCZ21}, where the authors proved the following.
\begin{proposition}\label{prop:BCZ}
The path $P_n$  has an irregular dominating labeling if and only if $n\geq4$ except for $n=6$.
\end{proposition}
We point out that the proof is constructive, but the result labeling is not necessarily optimal. Actually, the problem of establishing the exact value of $\gamma_i(P_n)$ is still open.
Previous proposition allows us to prove the following.

\begin{proposition}\label{prop:existencepath}
    Given $n\geq 1$, the path $P_n$ admits an extended irregular dominating labeling if and only if $n\neq2,3$.
\end{proposition}
\begin{proof}
  Since $P_2$ and $P_3$ do not have an extended irregular dominating labeling by Proposition \ref{prop:k12}, it remains to verify the converse.
  For $n=1$ the existence is trivial. An extended irregular dominating labeling for the paths  $P_{4}$ and $P_7$ is given by:
\[
\begin{aligned}
&P_{4}&: \quad& [3,1,0,2], \\  
&P_{7}&: \quad&[\square , 0, 2, 3, 1, \square , 4 ],
\end{aligned}
\] 
where $\square$ denotes a vertex with no label.
The other cases follow from the Remark \ref{rem:add0} and Proposition \ref{prop:BCZ}.    
\end{proof}

By Remark \ref{rem:add0} the extended irregular dominating labelings of previous proposition are, in general, not optimal.
In particular, in \cite{BCZ21},  a $k$-irregular dominating labeling of $P_n$ is constructed with $k=n-2$ for $n=7,8$, with $k=n-3$ for $n=9$, and with $k=n-4$ for every $n\geq10$. Starting for these labelings, we immediately have a $k'$-extended irregular dominating labeling of $P_{n+1}$ with $k'=k+1$. On the other hand, we believe that this result can be improved, that is that there exists a $k''$-extended irregular dominating labeling of $P_{n+1}$ for $k''<k'$.

In the remaining part of this section, we make some consideration on $\gamma_e(P_n)$, we start with a lemma whose proof is trivial.

\begin{lemma}\label{lemma:lowerbound}
For every $n\geq1$, $\gamma_e(P_n)\geq \lceil\frac{n+1}{2}\rceil$.
\end{lemma}

In the following result we establish when the equality holds in Lemma \ref{lemma:lowerbound}.

\begin{theorem}\label{thm:path}
Let $n\geq1$,
$\gamma_e(P_n)= \lceil\frac{n+1}{2}\rceil$  if and only if $n = 1,6,10$.
\end{theorem}
\begin{proof}
Suppose firstly $n$ odd, and set $n = 2m+1$ for some integer $m$. If $m = 0$ the result is trivial, so assume that $m\geq 1$. Note that $\gamma_e(P_{2m+1})=m+1$ implies that we have to use the labels from the set $[0,m]$, where every vertex labeled with an integer different from $0$ has to cover $2$ vertices. For this reason, the label $m$ has to be given to the vertex $x_{m+1}$, covering $x_1$ and $x_{2m+1}$.  However, it is not possible to assign the label $m-1$ and cover two vertices, hence an extended irregular dominating set of size $m+1$ cannot exist.
In other words, $\gamma_e(P_n)> \lceil\frac{n+1}{2}\rceil$.

Let now $n = 2m$ for some integer $m$. By Corollary \ref{cor:completo}, $P_2$ and $P_4$ do not admit an extended irregular dominating set, while for $P_6$ and $P_{10}$ we have the following labelings: 
 
\begin{center}
\includegraphics[width = \textwidth]{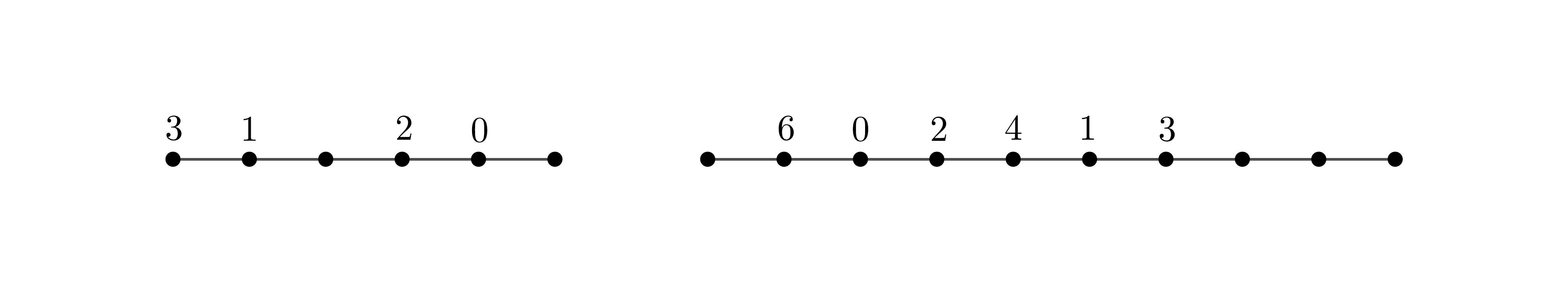}
\end{center}
Hence $\gamma_e(P_6)=4$ and  $\gamma_e(P_{10})=6$ hold.

Assume now that $m \geq 8$. Also here, $\gamma_e(P_{2m})=m+1$ implies we have to use labels in the set $[0,m-1]\cup{\bar\lambda}$, where $\bar\lambda \in [m,2m-1]$ in such a way that
the vertices receiving labels $0$ and $\bar\lambda$ cover exactly one vertex, and the vertices receiving labels from $[1,m-1]$ have to cover two vertices. To this aim, it is easy to see that  the label $m-1$ has to be assigned either to $x_m$ or to $x_{m+1}$. By symmetry we can assume without loss of generality that this label is assigned to $x_m$, thus covering $x_1$ and $x_{2m-1}$. This assignment forces the following labeling:
\begin{itemize}
	\item[$(1)$] the label $m-2$ has to be assigned to $x_{m+2}$, covering $x_4$ and $x_{2m}$; 
	\item[$(2)$] the label $m-3$ has to be assigned to $x_{m-1}$, dominating $x_2$  and $x_{2m-4}$;
	\item[$(3)$] the label $m-4$ has to be assigned to $x_{m+1}$, covering $x_5$ and $x_{2m-3}$.
\end{itemize} 
To summarize, up to this point, the vertices labeled are $\{x_{m-1},x_{m},x_{m+1},x_{m+2}\}$, and they dominate the set of vertices $\{x_1,x_{2},x_4,x_5\} \cup \{x_{2m-4},x_{2m-3},x_{2m-1},x_{2m}\}$, as shown in the following:
\begin{center}
\includegraphics[width = \textwidth]{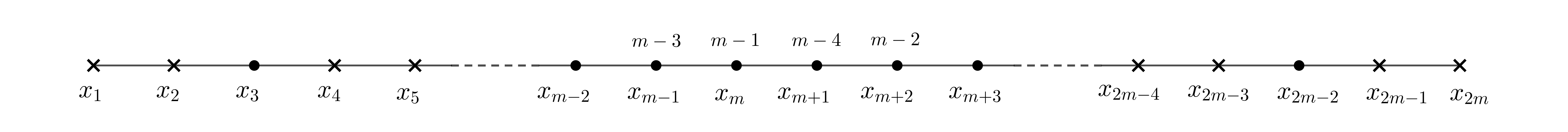}
\end{center}

It can then be seen that the label $m-5$ has to be assigned either to $x_{m-2}$ or $x_{m+3}$; by symmetry, it is not again restrictive to assume that this label is given to $x_{m-2}$, thus dominating $x_3$ and $x_{2m-7}$. Then, the label $m-6$ has to be given to $x_{m+4}$, covering $x_{10}$ and $x_{2m-2}$. However, it can be seen that now there is no possible assignment of the label $m-7$ to cover two of the remaining vertices, hence proving that for $m\geq 8$ there is no $k$-extended irregular dominating labeling of the path of even order with $k=\lceil\frac{n+1}{2}\rceil$.

To conclude, for $m=4,6,7$ repeat then the procedure shown before until the label $2$ is assigned. It can then be seen that it is not possible to assign the label $1$ to any vertex of the path.
\end{proof}

\begin{corollary}\label{cor:n+3/2}
Let $n\geq 4$ with $n\neq 6,10$, then $\gamma_e(P_n)\geq \lceil\frac{n+3}{2}\rceil.$
\end{corollary}
The next natural question is the following: when does the equality hold in Corollary \ref{cor:n+3/2}?

Clearly it holds for $n=4$, see the labeling of $P_4$ given in the proof of Proposition \ref{prop:existencepath}.
Below we show that the equality holds for every $n\in[5,26]\setminus\{6,10\}$. Some of these optimal extended irregular dominating labeling of $P_n$ have been obtained thanks to
to Remarks \ref{rem:removek} and \ref{rem:add0} starting from an irregular dominating labeling of $P_n$ constructed in \cite{Broe}. 
\[
\begin{aligned}
&P_{5}&: \quad& [3,1,\square,2,0], \\ 
&P_7&: \quad& 
[6,0,2,\square,1,3,\square] \\
&P_{8}&: \quad &
[\square,\square,3,1,4,2,0,6]\\
&P_{9}&: \quad &
[3,6,0,2,4,1,\square,\square,\square]\\
&P_{11}&: \quad& [\square,\square,5,3,1,4,2,\square,\square,7,0],\\
&P_{12}&: \quad&
[0,10,\square,5,3,1,4,2,\square,\square,7,\square], \\
&P_{13}&: \quad&
[\square,5,\square,\square,1,4,6,3,0,2,\square,9,\square], \\
&P_{14}&: \quad& [7,\square,\square,\square,2,4,6,3,5,1,\square,0,\square,8], \\  
&P_{15}&: \quad&[\square,\square,\square,\square,1,3,5,7,4,\square,0,2,\square,6,8 ],\\
&P_{16}&: \quad&[\square,9,\square,\square,1,3,5,7,4,\square,\square,2,\square,6,8,0 ],\\
&P_{17}&: \quad&[8,\square,\square,\square,2,5,3,\square,7,4,6,0,\square,1,\square,\square,9 ],\\
&P_{18}&: \quad&[8,10,\square,\square,2,5,3,\square,7,4,6,\square,\square,1,\square,\square,9,0 ],\\
&P_{19}&: \quad& [\square,\square,\square,13,\square,3,6,4,1,8,5,7,2,0,\square,\square,10,\square,\square],\\
&P_{20}&: \quad&
[\square,\square,\square,13,9,3,6,4,1,8,5,7,2,\square,\square,\square,10,\square,\square,0], \\
&P_{21}&:\quad&[\square,10,\square,\square,\square, 1, \square, 5,7,9,3,8,2,4,6,\square,0,12,\square,\square,\square] \\
&P_{22}&: \quad&[\square,\square,\square,\square,\square,\square,4,7,5,3,9,6,8,2,\square,\square,\square,1,11,10,12,0],\\
&P_{23}&: \quad&[12,\square,\square,\square,\square,3,0,4,8,10,5,7,9,1,6,2,\square,\square,\square,\square,11,\square],\\
&P_{24}&: \quad&[12,\square,\square,\square,\square,3,\square,4,8,10,5,7,9,1,6,2,\square,\square,\square,13,11,\square,0],\\
&P_{25}&: \quad&[\square,12,\square,\square,\square,\square,3,\square,4,8,10,5,7,9,1,6,2,\square,\square,\square,13,11,\square,0],\\
&P_{26}&: \quad&[14,12,\square,\square,\square,\square,4,1,\square,6,9,11,5,8,10,3,7,\square,2,\square,\square,\square,\square,\square,13,0],\\
\end{aligned}
\] 
where $\square$ denotes a vertex with no label.\\
\\

\section{Conclusions}
In this paper we have introduced the concept of an extended irregular dominating set focusing, in particular, on the optimal case.
In some cases we presented complete solution to existence problem of an optimal
extended irregular dominating set for some classes of graphs, while for other ones we have only some partial results.

For example note that in Section \ref{sec:cycles} about cycles, we have not considered the class of cycles $C_n$, with $n\equiv 0\pmod 4$.
A direct check shows that there exists no extended irregular dominating set of $C_n$ for $n=4,8$, while here we report an optimal extended irregular dominating labeling for the cycles $C_{4n}$ with $n \in [3,6]$:
\[
\begin{aligned}
&C_{12}&: \quad&(6,3,\square,\square,1,4,\square,5,0,2,\square,\square), \\  
&C_{16}&: \quad&(0,\square,\square,\square,\square,7,4,1,\square,\square,3,6,8,2,5,\square), \\
&C_{20}&: \quad&(10,8,\square,3,\square,9,1,\square,0,6,\square,7,\square,\square,\square,4,\square,5,\square,2), \\
&C_{24}&: \quad&(0,11,8,2,12,1,\square,\square,9,6,\square,\square,\square,\square,5,7,\square,4,\square,\square,\square,10,\square,3), \\
\end{aligned}
\] 
where the examples for $C_{16}$ and $C_{24}$ have been found with the aid of a computer by Falc\'{o}n \cite{F}.
At the moment we have no further results for this class of graphs, on the other hand for $n\in[3,6]$ we obtain really many labelings of $C_{4n}$ satisfying the required properties, hence we believe that $n=1,2$ are the only exception and we propose the following.
\begin{conjecture}
There exists an optimal extended irregular dominating labeling for $C_{4n}$ for every $n\geq3$.
\end{conjecture} 

About paths we established that $\gamma_e(P_n)= \lceil\frac{n+1}{2}\rceil$  if and only if $n = 1,6,10$, and hence that  $\gamma_e(P_n)\geq \lceil\frac{n+3}{2}\rceil$ for $n\geq 4$ with $n\neq 6,10$. Also we proved that the equality holds for every $n\in[4,26]\setminus\{6,10\}$. This leads us to propose another conjecture.
\begin{conjecture}
Let $n\geq 4$.
\[
\gamma_e(P_n) = \left\{ \begin{aligned}
&\left\lceil\frac{n+1}{2}\right\rceil \quad \text{ if $n=6,10$,} \\
&\left\lceil\frac{n+3}{2}\right\rceil \quad \text{otherwise.} \\ 
\end{aligned}\right.
\]
\end{conjecture}

\section*{Acknowledgements}
The authors are partially supported by INdAM - GNSAGA.
The authors would like to thank Prof. Giuseppe Mazzuoccolo for the fruitful discussion on the topic.

\end{document}